\documentclass[11pt]{preprint}

\title{The positive mass theorem under a spectral scalar\\ curvature bound on spin manifolds}

\author{Xiangsheng Wang}
\institute{School of Mathematics, Shandong University, Jinan, Shandong 250100, P.R. China\\ \email{xiangsheng@sdu.edu.cn}}
\date{\today}

\usepackage[english]{babel}
\usepackage[stretch=10,final]{microtype}
\usepackage[a4paper,left=2.5cm,right=2.5cm,top=4cm,bottom=2.5cm,marginparwidth=2cm]{geometry}
\usepackage[usenames,dvipsnames]{xcolor} \usepackage{titlesec}
\usepackage{enumitem}
\usepackage[olditem,oldenum]{paralist}
\usepackage{perpage} \usepackage[super]{nth}
\usepackage{subdepth}
\usepackage{booktabs}
\usepackage{todonotes}
\usepackage{fancyhdr}
\usepackage[title]{appendix}
\usepackage[colorlinks,linkcolor=blue,bookmarksnumbered,pagebackref,final]{hyperref}
\usepackage{graphicx}
\usepackage{subcaption}
\usepackage[]{lipsum}
\usepackage[alphabetic,initials,msc-links]{amsrefs}
\usepackage{derivative}
\usepackage{marginnote}

\usepackage{mathtools}
\usepackage{amsthm}
\usepackage{amssymb}
\usepackage{slashed} \usepackage{mleftright}
\usepackage{tikz-cd}
\usepackage{accents}
\usepackage{nicematrix} 

\usepackage[breakable, theorems, skins]{tcolorbox}
\tcbset{enhanced}

\usepackage[osf,theoremfont]{newtxtext} \usepackage[upint,smallerops]{newtxmath}
\usepackage[cal=euler,scr=rsfso,frak=euler]{mathalfa}
\usepackage[T1]{fontenc}

\usepackage[normalem]{ulem}

\newlength{\enumerateparindent}
\MakePerPage{footnote}

\DeclareSymbolFont{sfoperators}{OT1}{cmss}{m}{n}
\DeclareSymbolFontAlphabet{\mathsf}{sfoperators}
\makeatletter
\renewcommand{\operator@font}{\mathgroup\symsfoperators}
\makeatother

\numberwithin{equation}{section}
\swapnumbers                            \newtheoremstyle{plain}
{2ex plus 1ex minus .2ex}   {\medskipamount}   {\slshape}  {}       {\indent\bfseries\tlfstyle} {.}         {5pt plus 1pt minus 1pt} {}          

\newtheorem{theorem}[subsubsection]{Theorem}
\newtheorem{lemma}[subsubsection]{Lemma}

\newtheorem{proposition}[subsubsection]{Proposition}

\newtheorem*{claim*}{Claim}
\newtheorem*{lemma*}{Lemma}

\newtheoremstyle{definition}
{2ex plus 1ex minus .2ex}   {\medskipamount}   {}  {}       {\indent\bfseries\tlfstyle} {.}         {5pt plus 1pt minus 1pt} {}          \theoremstyle{definition}

\newtheorem{remark}[subsubsection]{Remark}

\newtheorem*{remark*}{Remark}
\newtheorem*{assumption*}{Assumption}
\newtheorem*{example*}{Example}

\makeatletter
\let\newtitle\@title
\let\newauthor\@author
\let\newdate\@date
\makeatother

\titleformat{\section}{\normalfont\large\bfseries\tlfstyle}{\thesection}{0.6em}{}
\titleformat{\subsection}[runin]{\normalfont\normalsize\bfseries\tlfstyle}{\thesubsection}{0.4em}{}
\titleformat{\subsubsection}[runin]{\normalfont\normalsize\bfseries\tlfstyle}{\indent\thesubsubsection}{0.4em}{}

\usetikzlibrary{shapes,arrows}
\tikzstyle{block} = [rectangle, draw, fill=blue!20, 
    text width=5em, text centered, rounded corners, minimum height=2em]
\tikzstyle{line} = [draw, -latex']
\tikzstyle{cloud} = [draw, ellipse,fill=red!20, node distance=3cm,
    minimum height=2em]
\tikzcdset{
  arrow style=tikz,
  diagrams={>={Straight Barb[scale=0.8]}}
}

\makeatletter
\newcommand*{\transpose}{{\mathpalette\@transpose{}}}
\newcommand*{\@transpose}[2]{\raisebox{\depth}{$\m@th#1\intercal$}}
\makeatother

\newcommand{\aset}[1]{\{#1\}}

\newcommand{\vol}{\mathrm{vol}}
\newcommand{\dv}{\,d{\vol}}
\newcommand{\dA}{\,d{\mathrm{A}}}
\newcommand{\rn}[1]{\textup{\uppercase\expandafter{\romannumeral#1}}}

\newcommand{\cC}{\mathcal{C}}

\newcommand{\cO}{\mathcal{O}}

\newcommand{\rL}{\mathrm{L}}

\newcommand{\rT}{\mathrm{T}}
\newcommand{\rW}{\mathrm{W}}

\newcommand{\bZ}{\mathbb{Z}}

\newcommand{\bR}{\mathbb{R}}

\newcommand{\sfS}{\mathsf{S}}
\newcommand{\sfE}{\mathsf{E}}

\newcommand{\wD}{\mathsf{D}_{\rho}}
\newcommand{\caE}{\opc^{\sfE}}
\newcommand{\DE}{\opD^{\sfE}}
\newcommand{\wDE}{\DE_{\rho}}
\newcommand{\dwDE}{\opD^{\sfE}_{\rho,\lambda, \kappa}}

\newcommand{\Mm}{M_{\mu}}

\newcommand{\norm}[1]{\left\Vert#1\right\Vert}

\renewcommand{\Re}{\operatorname{Re}}

\renewcommand{\emptyset}{\varnothing}

\DeclareMathOperator{\supp}{supp}

\DeclareMathOperator{\dist}{dist}

\DeclareMathOperator{\opc}{c}

\DeclareMathOperator{\opA}{A}
\DeclareMathOperator{\opD}{D}

\DeclareMathOperator{\opP}{P}

\DeclareMathOperator{\opB}{B}

\DeclareMathOperator{\opQ}{Q}

\DeclareMathOperator{\opV}{V}

\DeclareMathOperator{\scv}{scal}

\begin{document}
\maketitle
\begin{abstract}
  On spin manifolds, we give a proof of Brendle and Wang's refined positive mass theorem using the Dirac operator method.
  In the course of this proof, we discuss the origin of the special coefficient appearing in Brendle and Wang's spectral positivity condition for scalar curvature from the perspective of spin geometry.
\end{abstract}

\section{Introduction}
\subsection{}
Ever since the groundbreaking works of Schoen and Yau~\cite{Schoen_1979pp} and Witten~\cite{Witten_1981ne}, the positive mass theorem (PMT) has been a central topic in modern differential geometry.
Recently, the study of the PMT has undergone a wave of new developments.
Firstly, in~\cite{Bi_2026pr}, the authors prove the PMT up to dimension $19$ using a novel blow-up technique.
Shortly after that, in~\cite{Brendle_2026dh}, the authors prove the PMT for arbitrary dimension by combining the techniques of~\cite{Bi_2026pr} with some new ideas.
One of the key discoveries in~\cite{Brendle_2026dh} is a intriguing new spectral condition on scalar curvature.
However, the geometric meaning of this condition remains somewhat mysterious.
In this paper, we aim to investigate this condition from the spin geometry perspective.
As we will see, such a condition naturally emerges when obtaining the refined PMT on spin manifolds.

\subsection{}
We first review the spectral condition defined in~\cite{Brendle_2026dh}.
Let $(M,g)$ be a complete Riemannian manifold of dimension $n\ge 3$, containing an asymptotically Euclidean (AE) end $E_0$.
More precisely, $E_0$ is diffeomorphic to the complement of the unit ball in $\bR^n$.
Moreover, as in~\cite{Brendle_2026dh}, we assume that there exist real numbers $\alpha$ and $\delta > 0$ such that\footnote{In this paper, unless explicitly stated otherwise, the pointwise norm is calculated with respect to $g$.}
\begin{equation}
  \label{eq:def-g}
  | \bar{\nabla}^m(g - (1 + \alpha r^{2-n}) \bar{g})|_{\bar{g},x} \le C(m)r^{2-n-m- 2\delta}(x)
\end{equation}
holds at each point $x\in E_0$ for $m \in \bZ_{\ge 0}$, where $\bar{g}$ is the Euclidean metric on $E_0$, $\bar{\nabla}$ is the Levi-Civita connection of $\bar{g}$, and $r$ is the radial coordinate function on $E_0$ with respect to $\bar{g}$.

Let $\rho > 0$ be a smooth weight function on $M$ such that there exist real numbers $\beta$,
\begin{equation}
  \label{eq:def-rho}
  | \bar{\nabla}^m(\rho - (1 + \beta r^{2-n})) |_{\bar{g},x} \le C(m)r^{2-n-m- 2\delta}(x)
\end{equation}
holds at each point $x\in E_0$ for $m \in \bZ_{\ge 0}$.

We say that the scalar curvature $\scv$ of $g$ is \emph{positive in the BW sense} if and only if for any smooth test function $h$ such that $h|_{E_0}$ is constant at infinity and $h|_{M - E_0}$ vanishes outside a compact subset,\footnote{In this paper, we choose the sign convention such that $-\Delta$ is a semipositive operator.}
\begin{equation}
  \label{eq:bw-pos}
  \int_{M} |\nabla h|^2 \rho \dv + \frac{1}{2} \int_{M} (\scv - 2 \Delta \log \rho - \frac{n+1}{n+2} | \nabla \log \rho|^2 ) |h|^2 \rho\dv \ge \int_M Q |h|^2 \rho \dv,
\end{equation}
where $Q$ is a positive function on $M$ such that
\begin{equation}
  \label{eq:def-Q}
  | \bar{\nabla}^mQ |_{\bar{g},x} \le C(m)r^{-n-m- 2\delta}(x)
\end{equation}
holds at each point $x\in E_0$ for $m \in \bZ_{\ge 0}$.
We say that the scalar curvature $\scv$ of $g$ is \emph{nonnegative in the BW sense} if $Q = 0$ in (\ref{eq:bw-pos}).

In our opinion, the most baffling aspect in (\ref{eq:bw-pos}) is the presence of the coefficient $\frac{n+1}{n+2}$.
We discuss its origin from the perspective of spin geometry (see Remark~\ref{rk:param}).

\subsection{}
Under the condition (\ref{eq:bw-pos}), Brendle and Wang~\cite[Theorem~1.5]{Brendle_2026dh} prove a refined version of the PMT in any dimension.
In this paper, we prove a simplified version of their result under the spin condition.\footnote{Strictly speaking, in~\cite[Theorem~1.5]{Brendle_2026dh}, the authors only prove the strictly positive case.
We are not sure whether it is trivial or not to deduce the nonnegative case from the strictly positive case.}

\begin{theorem}
  \label{thm:bw}
If $M$ is spin and the scalar curvature $\scv$ of $g$ is positive (resp.\ nonnegative) in the BW sense, $(n-1)\alpha + 2\beta > 0$ (resp.\ $(n-1)\alpha + 2\beta \ge 0$).

\end{theorem}

\subsection{} To make the idea of the proof more transparent, and in particular to track the appearance of the coefficient $\frac{n+1}{n+2}$ more easily, in Section~\ref{sec:s-end}, we first prove Theorem~\ref{thm:bw} in the case where $M$ has only one end $E_0$.
Then, in Section~\ref{sec:m-end}, we explain how to extend the proof to the general case of multiple ends.

\subsection*{Acknowledgements.}
The author is very grateful for Prof.\ Weiping Zhang's inspiring comments on the draft of this paper.
The author was partially supported by NSFC grant 12471049 and 12101361, the Project of Young Scholars of Shandong University.

\section{Proof of Theorem~\ref{thm:bw}, single end case}\label{sec:s-end}
\subsection{}
In this section, we assume that $M$ has only one end, that is, $M - E_0$ is compact.

\subsection{Some function spaces.}
We extend the radial coordinate function $r$ on $E_0$ to a smooth positive function on $M^n$.
As in~\cite[Definition~A.20]{Lee_2019ge}, given any $p\ge 1$, $s\in \bR$, we define the weighted Lebesgue space with respect to $\rho$-measure $\rL^p_s(M, \rho \dv)$ to be the space of all functions $u \in \rL^p_{\mathrm{loc}}(M)$ with finite weighted norm
\begin{equation*}
  \norm{u}^p_{\rL^p_s(\rho)} = \int_M |u|^p r^{-sp-n} \rho\dv.
\end{equation*}
And for each positive integer $k$, we define the weighted Sobolev space with respect to $\rho$-measure $\rW^{k,p}_{s}(M, \rho \dv)$ to be the space of all functions $u \in \rW^{k,p}_{\mathrm{loc}}(M)$ with finite norm
\begin{equation*}
  \norm{u}^p_{\rW^{k,p}_s(\rho)} = \sum_{i=0}^k \| \nabla^i u\|^p_{\rL^p_{s-i} (\rho)}.
\end{equation*}

By taking $\rho = 1$, the above definitions reduce to the classical Lebesgue and Sobolev space, $\rL^p_s(M)$ and $\rW^{k,p}_s(M)$.
However, by the definition of $\rho$, (\ref{eq:def-rho}), $\rho$ is a bounded function on $M$.
As a result, the $\rL^p_s (\rho)$-norm is equivalent to the $\rL^p_s$-norm and $\rL^p_s(M, \rho\dv) = \rL^p_s(M)$, $\rW^{k,p}_{s}(M, \rho \dv) = \rW^{k,p}_{s}(M)$.

We can extend the above definition to sections of a vector bundle as usual.

Using the weighted Sobolev spaces, we can extend the test functions used in (\ref{eq:bw-pos}) to a larger class.

\begin{lemma}
  \label{lm:test-f}
Assume that $h$ is a function on $M$ such that $h = v_0 + h_c$, where $v_0 \in \bR$ is constant and $h_c\in \rW^{1,2}_{-(n-2)/2}(M, \rho \dv)$.
Then, (\ref{eq:bw-pos}) also holds for $h$.
  Moreover, let $F$ be a vector bundle associated with $\rT M$\footnote{In this paper, we take $F = \sfS(M)$ or $\sfS(M) \oplus \sfS(M)$, where $\sfS(M)$ is the spinor bundle of $M$.}.
  For any section $s \in \rW^{1,2}_{\mathrm{loc}}(F)$ such that $s = s_0 + s_c$, where $s_0|_{E_0}$ is a constant section with respect to $\bar{g}$ and $s_c\in \rW^{1,2}_{-(n-2)/2}(F, \rho \dv)$, (\ref{eq:bw-pos}) also holds for $|s|$.
\end{lemma}
\begin{proof}
  We first show the function case.
  Notice that if $h_c \in \cC^{\infty}_c(M)$, by definition, (\ref{eq:bw-pos}) holds for $h$.
  Fixing $v_0$, for $f \in \rW^{1,2}_{-(n-2)/2}(M, \rho\dv)$, we define a functional $\opA$ as
  \begin{equation*}
    \opA(f) = \int_{M} |\nabla (v_0 + f)|^2 \rho \dv + \frac{1}{2} \int_{M} (\scv - 2 \Delta \log \rho - \frac{n+1}{n+2} | \nabla \log \rho|^2 - Q ) |v_0 + f|^2 \rho\dv.
  \end{equation*}
  By (\ref{eq:def-g}), (\ref{eq:def-rho}) and (\ref{eq:def-Q}) and H\"older's inequality, $\opA$ is a continuous functional with respect to the $\rW^{1,2}_{-(n-2)/2}(\rho)$-norm.
  Since $\cC^{\infty}_c(M)$ is dense in $\rW^{1,2}_{-(n-2)/2}(M, \rho\dv)$, the desired result follows from the compact support case.

For the vector bundle case, we first notice that if $s_0|_{E_0}$ is a constant section with respect to $\bar{g}$, by (\ref{eq:def-g}),
  \begin{equation}
    \label{eq:norm-s0}
    |s_0| = v_0 + \cO(r^{-(n-2)}),
  \end{equation}
  where $v_0$ is a constant.

  Next, we consider the integrability of $|s| - v_0$.
  Since $s= s_0 + s_c$,
  \begin{equation*}
    \frac{(|s| - v_0)^2}{r^2} = \frac{(|s| - |s_0|)(|s| + |s_0| - 2v_0)}{r^2} + \frac{(|s_0| - v_0)^2}{r^2}
    \le \frac{|s_c|^2 + 2|s_c|||s_0| - v_0|}{r^2} + \frac{(|s_0| - v_0)^2}{r^2}.
  \end{equation*}
Since $s_c \in \rW^{1,2}_{-(n-2)/2}(F, \rho \dv)$, together with (\ref{eq:norm-s0}), we know from the above estimate that $|s| - v_0 \in \rL^{2}_{-(n-2)/2}(M, \rho \dv)$.

  Moreover, since $s_0|_{E_0}$ is a constant section and $s_c \in \rW^{1,2}_{-(n-2)/2}(F, \rho \dv)$, by (\ref{eq:def-g}), we know $|\nabla s| \in \rL^2(M, \rho \dv)$.
  Therefore, by Kato's inequality again,
  \begin{equation*}
    \int_M|\nabla(|s| - v_0)|^2 \rho \dv = \int_M|\nabla|s||^2 \rho \dv \le \int_M|\nabla s|^2 \rho \dv < +\infty.
  \end{equation*}
  Consequently, we have verified that $|s| - v_0 \in \rW^{1,2}_{-(n-2)/2}(M, \rho \dv)$.
  Finally, applying the proved function case to $|s|$, we show that (\ref{eq:bw-pos}) also holds for $|s|$.
\end{proof}

\subsection{Idea of the proof.}\label{sub:idea}
As expected, the proof is merely another adaptation of Witten's classical work~\cite{Witten_1981ne}.
As preparation, we first recall the relation between the mass of $g$ and the quantities appearing in Theorem~\ref{thm:bw}.

Let $\sfS(M)$ be the spinor bundle of $M$ and
\begin{equation*}
  \opD \coloneqq \sum_{i=1}^n \opc(e_i)\nabla_{e_i}: \Gamma(M,\sfS(M)) \rightarrow \Gamma(M,\sfS(M))
\end{equation*}
be the Dirac operator on $M$, where $e_1,\cdots,e_n$ is a locally othornormal basis for $\rT M$ and $\opc(-)$ is the Clifford action on $\sfS(M)$.
Take $\Psi$ to be a Witten spinor.
In other words,
\begin{equation}
  \label{eq:witten-sp}
  \opD \Psi = 0 \text{ and }\Psi - \Psi_{\infty} \in \rW^{1,2}_{-(n-2)/2}(\sfS(M)),
\end{equation}
where $\Psi_{\infty}\in \Gamma(M, \sfS(M))$ and $\Psi_{\infty}|_{E_0}$ is a constant spinor of unit norm with respect to the Euclidean metric.

Since $\Psi$ is a harmonic spinor, by the Lichnerowicz formula and the divergence theorem, we have
\begin{equation}
  \label{eq:m-int}
  \int_{S_r} \frac{1}{2} \nu(|\Psi|^2) \dA = \int_{M_r} (|\nabla \Psi|^2 + \frac{1}{4} \scv |\Psi|^2) \dv,
\end{equation}
where $S_r \subseteq E_0$ is the Euclidean sphere of radius $r$, $M_r \subseteq M$ is the manifold with boundary enclosed by $S_r$, and $\nu$ is the outward-pointing unit normal vector field on $S_r$ with respect to $g$.
Following the convention in~\cite{Lee_2019ge}, the mass of $g$ is defined (or proved) to be
\begin{equation}
  \label{eq:def-mass}
  m = \frac{1}{(n-1)\omega_{n-1}} \lim_{r\rightarrow \infty} \int_{S_r} \nu(|\Psi|^2) \dA,
\end{equation}
where $\omega_{n-1}$ is the volume of the standard $(n-1)$-dimensional sphere.
With this definition of the mass, by (\ref{eq:def-g}), we can verify that
\begin{equation}
  \label{eq:alpha-m}
  m = \frac{n-2}{2}\alpha.
\end{equation}
Let $\psi = \rho^{-1/2}\Psi$.
By (\ref{eq:def-rho}), (\ref{eq:def-mass}) and (\ref{eq:alpha-m}), we have
\begin{equation}
  \label{eq:a-b}
  \frac{1}{(n-1)\omega_{n-1}} \lim_{r\rightarrow \infty} \int_{S_r} \nu(|\psi|^2) \rho \dA = m + \frac{n-2}{n-1}\beta = \frac{n-2}{2(n-1)}((n-1)\alpha + 2\beta).
\end{equation}
Therefore, to prove Theorem~\ref{thm:bw}, it suffices to show the left hand side of (\ref{eq:a-b}) is positive (resp.\ nonnegative) by using a formula similar to (\ref{eq:m-int}) and the condition (\ref{eq:bw-pos}).

However, there exists a serious gap in this outlined approach; that is, the existence of the Witten spinor $\Psi$ cannot be taken for granted in the current setting.
Recall that in Witten's proof~\cite[\S~5]{Lee_2019ge}, the existence of the Witten spinor $\Psi$ depends on the pointwise positivity of the scalar curvature.
But now we only have a spectral bound on the scalar curvature.
Consequently, in the following, we first prove the existence of $\psi$ and then construct $\Psi$ using $\psi$.

\subsection{}
Let $\wD$ be the weighted Dirac operator with respect to $\rho$,
\begin{equation}
  \label{eq:def-wD}
  \wD \coloneqq \rho^{-1/2} \opD \rho^{1/2} = \opD + \frac{1}{2}\opc(\nabla \log \rho).
\end{equation}

Recall that $\wD$ is formally self-adjoint for the Hilbert space $\rL^2(\sfS(M), \rho \dv)$.
Moreover, $\wD$ satisfies the weighted Lichnerowicz formula,
\begin{equation*}
  \wD^2 = - \Delta^{\sfS} - \nabla_{\nabla \log \rho} + \frac{1}{4} (\scv - 2 \Delta \log \rho - |\nabla \log \rho|^2).
\end{equation*}
For any smooth spinor $s\in \Gamma(M_r, \sfS(M))$, the above formula, together with the divergence theorem, gives
\begin{equation}
  \label{eq:bd-term}
  \int_{S_r} \langle s, \nabla_{\nu}s + \opc(\nu) \wD s \rangle \rho\dA = \int_{M_r} \big(|\nabla s|^2 - |\wD s|^2 + \frac{1}{4}  (\scv - 2 \Delta \log \rho - |\nabla \log \rho|^2) |s|^2 \big)\rho\dv.
\end{equation}

As we have said, we need the following existence result for a Witten-type spinor for $\wD$.

\begin{proposition}
  \label{prop:ex-w-spinor}
  There exists a smooth spinor $\psi\in \Gamma(M, \sfS(M))$ such that
  $\psi$ satisfies properties similar to Witten's spinor,
  \begin{equation*}
    \label{eq:w-witten-sp}
    \wD \psi = 0 \text{ and }\psi - \Psi_{\infty} \in \rW^{1,2}_{-(n-2)/2}(\sfS(M)).
  \end{equation*}
\end{proposition}

Once Proposition~\ref{prop:ex-w-spinor} is proved, we obtain a formula from (\ref{eq:bd-term}) similar to (\ref{eq:m-int}) on the weighted manifold,
\begin{equation}
  \label{eq:m-int-w}
  \int_{S_r} \frac{1}{2} \nu(|\psi|^2) \rho\dA = \int_{M_r} \big(|\nabla \psi|^2  + \frac{1}{4}  (\scv - 2 \Delta \log \rho - |\nabla \log \rho|^2) |\psi|^2 \big)\rho\dv.
\end{equation}
Moreover, by setting
\begin{equation}
  \label{eq:def-Psi}
  \Psi = \rho^{1/2}\psi,
\end{equation}
$\Psi$ satisfies (\ref{eq:witten-sp}) due to (\ref{eq:def-rho}).
In particular, (\ref{eq:m-int})-(\ref{eq:a-b}) have been established.
Therefore, combining (\ref{eq:a-b}) with (\ref{eq:m-int-w}), to prove Theorem~\ref{thm:bw}, the remaining task is to show the right hand side of (\ref{eq:m-int-w}) is positive (resp.\ nonnegative).

See~\cite{Balfauf_2022sm,Chu_2024no} for more information on the Dirac operator and the mass on weighted manifolds.

\subsection{A Kato-type inequality.}
By the definition of $\opD$ and $\rho$, one can check that
\begin{equation}
  \label{eq:wd-isom}
  \wD: \rW^{1,2}_{-(n-2)/2}(\sfS(M),\rho \dv) \rightarrow \rL^2(\sfS(M), \rho \dv)
\end{equation}
is a well-defined bounded linear map.
Note that $\rL^2 = \rL^2_{-n/2} = \rL^2_{-(n-2)/2 - 1} $.
As preparation to prove Proposition~\ref{prop:ex-w-spinor}, we first prove a Kato-type inequality for $s\in \rW^{1,2}_{-(n-2)/2}(\sfS(M)) \cap \Gamma(M, \sfS(M))$.

Let $\bar{s} = \rho^{1/2}s$.
By~\cite[Proposition~3.3]{Bei_2026ge}, at any point $x$ where $s(x) \ne 0$, the following pointwise equality holds,
\begin{equation}
  \label{eq:bc-eq}
  | \nabla |\bar{s}||^2_x = \frac{n-1}{n} |\nabla \bar{s} |^2_x - \frac{n-1}{n}|P \bar{s}|^2_x + \frac{2}{n} \Re \langle\opD \bar{s}, \opc(\nabla \log |\bar{s}|) \bar{s}\rangle_x,
\end{equation}
where $P$ is a differential operator whose precise form is irrelevant for our purposes.
Now, let $0 < l < 1$ be a parameter to be determined.
By elementary inequalities, the above equality implies that
\begin{equation*}
  | \nabla |\bar{s}||^2_x \le \frac{n-1}{n} |\nabla \bar{s} |^2_x  + \frac{2}{n} \Re \langle\opD \bar{s}, \opc(\nabla \log |\bar{s}|) \bar{s}\rangle_x \le \frac{n-1}{n} |\nabla \bar{s} |^2_x + \frac{1}{nl} |\opD \bar{s}|^2_x + \frac{l}{n} |\nabla |\bar{s}||^2_x,
\end{equation*}
or equivalently
\begin{equation}
  \label{eq:kato-extra}
  (1 - \frac{l}{n})| \nabla |\bar{s}||^2_x \le \frac{n-1}{n} |\nabla \bar{s} |^2_x + \frac{1}{nl} |\opD \bar{s}|^2_x.
\end{equation}
Note that
\begin{equation}
  \label{eq:bs-sqrt}
  \begin{gathered}
    |\nabla \bar{s} |^2_x = \rho (\frac{1}{4}|\nabla \log \rho|^2_x |s|^2_x + \frac{1}{2} \langle\nabla \log \rho, \nabla |s|^2\rangle_x + |\nabla s|^2_x),\\
    |\nabla |\bar{s}| |^2_x = \rho (\frac{1}{4}|\nabla \log \rho|^2_x |s|^2_x + \frac{1}{2} \langle\nabla \log \rho, \nabla |s|^2\rangle_x + |\nabla |s||^2_x).
  \end{gathered}
\end{equation}
By (\ref{eq:def-wD}), (\ref{eq:kato-extra}) and (\ref{eq:bs-sqrt}), we have
\begin{equation*}
  (1 - \frac{l}{n})| \nabla |{s}||^2_x \le \frac{n-1}{n} |\nabla {s} |^2_x  - \frac{1-l}{2n} \langle\nabla \log \rho, \nabla |s|^2 \rangle_x - \frac{1-l}{4n} |\log \rho|^2_x |s|^2_x + \frac{1}{nl} |\opD_{\rho} \bar{s}|^2_x.
\end{equation*}
By elementary inequality, for $k > 0$,
\begin{equation*}
  |\langle\nabla \log \rho, \nabla |s|^2 \rangle|_x \le 2|\nabla \log \rho|_x |s|_x |\nabla|s||_x \le k |\nabla \log \rho|^2_x |s|^2_x + \frac{1}{k} |\nabla | s||^2_x.
\end{equation*}
The above two inequalities imply that
\begin{equation}
  \label{eq:key-est-s}
  |\nabla s|^2_x \ge \frac{n}{n-1} (1 - \frac{l}{n} - \frac{1-l}{2n k})|\nabla |s||^2_x + \frac{n}{n-1} (\frac{1-l}{4n} - \frac{k(1-l)}{2n}) |\nabla \log \rho|^2_x |s|^2_x - \frac{1}{(n-1)l} |\wD s|^2_x.
\end{equation}

We choose $k$ such that
\begin{equation}
  \label{eq:k-range-s}
  \begin{cases}
    \frac{n}{n-1} (1 - \frac{l}{n} - \frac{1-l}{2n k}) > \frac{1}{2},\\
    \frac{n}{n-1} (\frac{1-l}{4n} - \frac{k(1-l)}{2n}) \ge \frac{1}{4(n+2)},
  \end{cases}
\end{equation}
which is equivalent to
\begin{equation}
  \label{eq:k-ineq}
  \frac{1-l}{n - 2l +1} < k \le \frac{1}{2} - \frac{n-1}{2(1-l)(n+2)}.
\end{equation}
Therefore, $k$ exists if and only if
\begin{equation*}
  \frac{1-l}{n - 2l +1} < \frac{1}{2} - \frac{n-1}{2(1-l)(n+2)},
\end{equation*}
which is equivalent to
\begin{equation}
  \label{eq:l-ineq}
  1 < n < \frac{1}{l}.
\end{equation}
Fix $k$ and $l$ such that (\ref{eq:k-ineq}) and (\ref{eq:l-ineq}) hold.
Then, by (\ref{eq:key-est-s}), we can find a constant $c_0$ (only depending on $n$) such that
\begin{equation}
  \label{eq:pt-kato-s}
  |\nabla s|^2_x \ge \frac{1 + c_0}{2}|\nabla |s||^2_x + \frac{1}{4(n+2)}|\nabla \log \rho|^2_x |s|^2_x - \frac{1}{(n-1)l} |\wD s|^2_x.
\end{equation}
It is well known that $\nabla s = 0$ holds almost everywhere on the zero set of $s$,~\cite[Lemma~7.7]{Gilbarg_2001el}.
As a result, integrating (\ref{eq:pt-kato-s}) over $M$ yields
\begin{equation}
  \label{eq:kato-extra-int}
  \int_M|\nabla s|^2 \rho\dv \ge \frac{1 + c_0}{2}\int_M |\nabla |s||^2 \rho \dv + \int_M\Big(\frac{1}{4(n+2)}|\nabla \log \rho|^2 |s|^2 - \frac{1}{(n-1)l} |\wD s|^2 \Big) \rho\dv.
\end{equation}

\begin{remark}
  \label{rk:kato-a}
  Although we do not need it, (\ref{eq:pt-kato-s}) can be improved slightly as follows.
  We can introduce another small positive parameter $0< a < 1$, which depends only on $n$, such that the following variant of (\ref{eq:k-range-s}) holds,
  \begin{equation}
    \label{eq:k-range-s-a}
    \begin{cases}
      \frac{n}{n-1+a} (1 - \frac{l}{n} - \frac{1-l}{2n k}) > \frac{1}{2},\\
      \frac{n}{n-1+a} (\frac{1-l}{4n} - \frac{k(1-l)}{2n}) > \frac{1}{4(n+2)}.
    \end{cases}
  \end{equation}
  The same calculation shows that $k$ exists for (\ref{eq:k-range-s-a}) if and only if $n$ satisfies
  \begin{equation*}
    n_{-} < n < n_+,\text{ where } n_{\pm} = \frac{1 + l - a(1-l) \pm \sqrt{(a(1+l) - (1-l))^2 - 16al(1-l)}}{2l}.
  \end{equation*}
  Fix $n$.
  If we take $a = 0$, $l = 1/(2n)$, i.e.\ $n_- = 1$ and $n_+ = 2n$, the above inequality holds.
  In fact, the above inequality reduces to a special case of (\ref{eq:l-ineq}) in this setting.
  Based on this special case, and by the continuity of $n_{\pm}$ with respect to $a,l$, the above inequality still holds for $l=1/(2n)$ and $a$ sufficiently small.

  Fix such $a$ and $l$.
  (\ref{eq:key-est-s}) and (\ref{eq:k-range-s-a}) imply that there exists a constant $0 < c_0' < 1$ (depending only on $n$) such that a stronger form of (\ref{eq:pt-kato-s}) holds,
  \begin{equation*}
|\nabla s|^2_x \ge c_0'|\nabla s|^2 +  \frac{1}{2}|\nabla |s||^2_x + \frac{1}{4(n+2)}|\nabla \log \rho|^2_x |s|^2_x - \frac{1}{(n-1)l} |\wD s|^2_x.
  \end{equation*}
\end{remark}

\subsection{Proof of Proposition~\ref{prop:ex-w-spinor}.}
We first show that (\ref{eq:wd-isom}) is an isomorphism.

For $s\in \rW^{1,2}_{-(n-2)/2}(\sfS(M), \rho\dv) \cap \Gamma(M, \sfS(M))$, by~\cite[Corollary~5.15]{Lee_2019ge} or a direct calculation, the left hand side of (\ref{eq:bd-term}) vanishes as $r \rightarrow \infty$.
Therefore, by (\ref{eq:bd-term}),
\begin{equation}
  \label{eq:wD-int-eq}
  \int_M |\wD s|^2 \rho\dv =  \int_M\big(|\nabla s|^2 + \frac{1}{4}  (\scv - 2 \Delta \log \rho - |\nabla \log \rho|^2) |s|^2 \big)\rho\dv.
\end{equation}
By (\ref{eq:kato-extra-int}), the above equality implies that
\begin{multline}
  \label{eq:est-wD}
  (1 + \frac{1}{(n-1)l})\int_M |\wD s|^2 \rho\dv \ge \frac{c_0}{2} \int_M |\nabla |s||^2 \rho \dv \\
  + \frac{1}{2} \Big(\int_M |\nabla |s||^2 \rho \dv
  + \frac{1}{2} \int_{M} (\scv - 2 \Delta \log \rho - \frac{n+1}{n+2} | \nabla \log \rho|^2 ) |s|^2 \rho\dv \Big).
\end{multline}
Due to (\ref{eq:bw-pos}) and Lemma~\ref{lm:test-f}, the last term in the above inequality is nonnegative.
Consequently, there exists a constant $c_1> 0$ such that\footnote{In this section, unless explicitly stated otherwise, the norm $\norm{\bullet}$ is evaluated on $M$.}
\begin{equation*}
  \norm{\nabla|s|}^2_{\rL^2(\rho)} = \int_M |\nabla |s||^2 \rho \dv \le c_1 \int_M |\wD s|^2 \rho\dv = c_1\norm{\wD s}^2_{\rL^2(\rho)}.
\end{equation*}

On the other hand, by the weighted Poincar\'e inequality~\cite[Theorem~A.28]{Lee_2019ge}, there exists a constant $c_2> 0$ such that\footnote{If we apply the weighted Poincar\'e inequality directly as the form given in~\cite[Theorem~A.28]{Lee_2019ge}, the right hand side of the following inequality should be $c'\|{\nabla(\rho^{1/2}|s|)}\|^2_{\rL^2}$, from which the form presented here follows by elementary inequalities.}
\begin{equation*}
  \norm{s}^2_{\rL^{2}_{-(n-2)/2}(\rho)} \le c_2 \norm{\nabla|s|}^2_{\rL^2(\rho)}.
\end{equation*}
By the above two inequalities, we have
\begin{equation}
  \label{eq:l2-d-norm}
  \norm{s}^2_{\rL^{2}_{-(n-2)/2}(\rho)} \le c_1c_2 \norm{\wD s}^2_{\rL^2(\rho)}.
\end{equation}

Meanwhile, by (\ref{eq:def-g}) and (\ref{eq:def-rho}), there also exists a constant $c_3 > 0$ such that
\begin{equation*}
  \int_{M}(\scv - 2 \Delta \log \rho - |\nabla \log \rho|^2) |s|^2 \rho\dv \ge - c_3 \norm{ s}^2_{\rL^{2}_{-(n-2)/2}(\rho)}.
\end{equation*}
Then, by (\ref{eq:wD-int-eq}) and (\ref{eq:l2-d-norm}),
\begin{equation}
  \label{eq:h1-d-norm}
  \norm{\nabla s}^2_{\rL^2(\rho)} = \int_{M} |\nabla s|^2 \rho\dv \le (1 + \frac{c_1c_2c_3}{4})\norm{\wD s}^2_{\rL^2(\rho)}.
\end{equation}

Combining (\ref{eq:l2-d-norm}) and (\ref{eq:h1-d-norm}), we obtain the following injectivity estimate,
\begin{equation}
  \label{eq:inj-est}
  \norm{s}^2_{\rW^{1,2}_{-(n-2)/2}(\rho)} \le \norm{\nabla s}^2_{\rL^2(\rho)} + \norm{s}^2_{\rL^{2}_{-(n-2)/2}(\rho)}
\le C \norm{\wD s}^2_{\rL^2(\rho)},
\end{equation}
where $s\in \rW^{1,2}_{-(n-2)/2}(\sfS(M), \rho\dv) \cap \Gamma(M, \sfS(M))$ and $C> 0$ is a constant independent of $s$.
By passing to a limit, we know that (\ref{eq:inj-est}) also holds for any $s\in \rW^{1,2}_{-(n-2)/2}(\sfS(M),\rho \dv)$.

By (\ref{eq:inj-est}), the map (\ref{eq:wd-isom}) is injective.
Since $\wD$ is a formally self-adjoint operator on $\rL^2(\sfS(M),\allowbreak \rho \dv)$, by the same argument as in~\cite[Proposition~5.6]{Lee_2019ge}, the surjectivity of (\ref{eq:wd-isom}) also follows from (\ref{eq:inj-est}).
Therefore, we have shown that (\ref{eq:wd-isom}) is an isomorphism.

Now, by (\ref{eq:def-g}) and (\ref{eq:def-rho}), $\xi \coloneqq \wD \Psi_{\infty} \in \rL^2(\sfS(M), \rho \dv)$.
Using the isomorphism (\ref{eq:wd-isom}), we can find $\eta \in \rW^{1,2}_{-(n-2)/2}(\sfS(M),\rho \dv)$ such that $-\xi = \wD \eta$.
Setting $\psi = \eta + \Psi_{\infty}$, we have $\wD \psi = 0$ and $\psi - \Psi_{\infty} \in \rW^{1,2}_{-(n-2)/2}(\sfS(M),\rho \dv) = \rW^{1,2}_{-(n-2)/2}(\sfS(M))$.
Finally, $\psi$ is smooth due to the elliptic regularity of $\wD$.

\subsection{}\label{sub:kato}
For the Witten-type spinor $\psi$, we are going to show the following Kato-type inequality similar to (\ref{eq:kato-extra-int}),
\begin{equation}
  \label{eq:kato}
  \int_{M_r} \bigl( |\nabla \psi|^2 - \frac{1}{4(n+2)} |\nabla \log \rho|^2 |\psi|^2 \bigr) \rho \dv \ge \frac{1}{2} \int_{M_r} |\nabla |\psi||^2 \rho \dv.
\end{equation}
We will prove the above inequality in the same way as (\ref{eq:kato-extra-int}).
As before, it suffices to show the corresponding pointwise inequality at any point $x$ where $\psi(x) \ne 0$.

Recall that in (\ref{eq:def-Psi}), we have defined $\Psi = \rho^{1/2}\psi$ such that $\opD \Psi = 0$.
As a result, by the refined Kato inequality~\cite[(3.9)]{Calderbank_2000re} and~\cite[Proposition~4.1]{Davaux_2003op} (or directly by (\ref{eq:bc-eq})),
\begin{equation}
  \label{eq:r-kato}
  |\nabla | \Psi| |^2_x \le \frac{n-1}{n} |\nabla \Psi|^2_x.
\end{equation}
Since $\Psi = \rho^{1/2}\psi$, combining (\ref{eq:r-kato}) and (\ref{eq:bs-sqrt}), we have
\begin{equation*}
  |\nabla |\psi| |^2_x \le \frac{n-1}{n} |\nabla \psi|^2_x - \frac{1}{2n} \langle\nabla \log \rho, \nabla |\psi|^2 \rangle_x - \frac{1}{4n} |\log \rho|^2_x |\psi|^2_x.
\end{equation*}
By elementary inequalities, for $k > 0$,
\begin{equation*}
  |\langle\nabla \log \rho, \nabla |\psi|^2 \rangle|_x \le 2|\nabla \log \rho|_x |\psi|_x |\nabla|\psi||_x \le k |\nabla \log \rho|^2_x |\psi|^2_x + \frac{1}{k} |\nabla | \psi||^2_x.
\end{equation*}
The above two inequalities imply that
\begin{equation}
  \label{eq:key-est}
  |\nabla \psi|^2_x \ge \frac{n}{n-1} (1 - \frac{1}{2n k})|\nabla |\psi||^2_x + \frac{n}{n-1} (\frac{1}{4n} - \frac{k}{2n}) |\nabla \log \rho|^2_x |\psi|^2_x.
\end{equation}

We choose $k$ such that
\begin{equation}
  \label{eq:k-range}
  \begin{cases}
    \frac{n}{n-1} (1 - \frac{1}{2n k}) \ge \frac{1}{2},\\
    \frac{n}{n-1} (\frac{1}{4n} - \frac{k}{2n}) \ge \frac{1}{4(n+2)},
  \end{cases}
\end{equation}
which is equivalent to
\begin{equation*}
  \frac{1}{n+1} \le k \le \frac{3}{2(n+2)}.
\end{equation*}
Therefore, $k$ exists if and only if
\begin{equation*}
  \frac{1}{n+1} \le \frac{3}{2(n+2)},\text{ i.e.\ }n \ge 1.
\end{equation*}
As a result, by (\ref{eq:k-range}),
\begin{equation}
  \label{eq:pt-kato}
  |\nabla \psi|^2_x \ge \frac{1}{2}|\nabla |\psi||^2_x + \frac{1}{4(n+2)}|\nabla \log \rho|^2_x |\psi|^2_x,
\end{equation}
which is the pointwise estimate required for (\ref{eq:kato}).

\subsection{}
With (\ref{eq:kato}), and using (\ref{eq:a-b}), (\ref{eq:m-int-w}), the following inequality holds,
\begin{equation*}
  \label{eq:main-est}
  \begin{multlined}
    \frac{\omega_{n-1}(n-2)}{2}((n-1)\alpha + 2\beta) =  \int_{M} \big(|\nabla \psi|^2  + \frac{1}{4}  (\scv - 2 \Delta \log \rho - |\nabla \log \rho|^2) |\psi|^2 \big)\rho\dv\\
    =  \int_{M} \bigl( |\nabla \psi|^2 - \frac{1}{4(n+2)} |\nabla \log \rho|^2 |\psi|^2 \bigr) \rho \dv + \frac{1}{4} \int_{M} (\scv - 2 \Delta \log \rho - \frac{n+1}{n+2} | \nabla \log \rho|^2 ) |\psi|^2 \rho\dv \\
    \ge  \frac{1}{2} \Big(\int_{M} |\nabla |\psi||^2 \rho \dv + \frac{1}{2} \int_{M} (\scv - 2 \Delta \log \rho - \frac{n+1}{n+2} | \nabla \log \rho|^2 ) |\psi|^2 \rho\dv \Big).
  \end{multlined}
\end{equation*}
Due to Lemma~\ref{lm:test-f}, we can take $h = |\psi|$ in (\ref{eq:bw-pos}).
Then, $(n-1)\alpha + 2\beta > 0$ (resp.\ $(n-1)\alpha + 2\beta \ge 0$) follows from the above inequality and (\ref{eq:bw-pos}) immediately.

The proof of Theorem~\ref{thm:bw} in the single end case is complete.

\begin{remark}
  \label{rk:param}

  After finishing the proof of Theorem~\ref{thm:bw} in the single end case, now we reexamine the role played by the coefficient $\frac{n+1}{n+2}$ in (\ref{eq:bw-pos}) during the proof.

  To make the analysis more transparent, we replace the coefficient $\frac{n+1}{n+2}$ in (\ref{eq:bw-pos}) with a parameter $\gamma$.\footnote{A similar parameter also appears in~\cite{Wang_2026lm}.}
  Retracing the proof, we see that to prove Theorem~\ref{thm:bw} in this more general setting, we have to show the following variant of (\ref{eq:pt-kato}),
\begin{equation}
    \label{eq:pt-kato-g}
    |\nabla \psi|^2_x \ge \frac{1}{2}|\nabla |\psi||^2_x + \frac{1-\gamma}{4}|\nabla \log \rho|^2_x |\psi|^2_x.
  \end{equation}

  If $\gamma \ge 1$, the above inequality follows trivially from the usual Kato inequality.
  If $\gamma < 1$, a calculation analogous to (\ref{eq:r-kato})-(\ref{eq:k-range}) yields the following constraint on $k$,
  \begin{equation*}
    \frac{1}{n+1} \le k \le \frac{1 - (1-\gamma)(n-1)}{2},
  \end{equation*}
  which implies that $\gamma$ should satisfy
  \begin{equation}
    \label{eq:gamma-r1}
    \gamma \ge \frac{n}{n+1},
  \end{equation}
  to ensure the existence of $k$.
  Therefore, to obtain the estimate (\ref{eq:pt-kato-g}), it seems that $\gamma= \frac{n}{n+1}$, rather than $\gamma = \frac{n+1}{n+2}$, is the optimal value for $\gamma$.

  However, this is not the end of the story.
  Recall that to obtain the estimate (\ref{eq:pt-kato}), we should first verify the existence of the Witten-type spinor $\psi$.
  The key to such an existence result is the estimate (\ref{eq:pt-kato-s}).
  To prove the corresponding estimate (\ref{eq:pt-kato-s}) with the extra parameter $\gamma$ as above, a calculation similar to (\ref{eq:bc-eq})-(\ref{eq:l-ineq}) shows that $\gamma$ must satisfy
\begin{equation}
    \label{eq:gamma-r3}
\frac{n-1}{1-l} < \frac{2\gamma - 1}{1 - \gamma}.
  \end{equation}
  Since $0 < l < 1$, the above inequality yields the following range on $\gamma$,
  \begin{equation}
    \label{eq:gamma-r2}
    n-1 < \frac{2\gamma - 1}{1 - \gamma},\text{ or equivalently } \gamma > \frac{n}{n+1}.
  \end{equation}
Combining (\ref{eq:gamma-r1}) and (\ref{eq:gamma-r2}), we observe that the admissible range of $\gamma$ is in fact an open interval.
  In other words, there is no optimal choice for $\gamma$.
  Among all admissible values, $\gamma = \frac{n+1}{n+2}$ chosen in (\ref{eq:bw-pos}) has the merit of a relatively simple form.
  In fact, the right hand side of (\ref{eq:gamma-r3}) is equal to $n$ with such a value.

  Finally, based on the above discussion, it would be interesting to ask: when $\gamma = \frac{n}{n+1}$, does the Witten-type spinor $\psi$ still exist?
\end{remark}

\section{Proof of Theorem~\ref{thm:bw}, multiple ends case}\label{sec:m-end}
\subsection{}
In this section, we prove Theorem~\ref{thm:bw} for the general case, that is, $(M,g)$ is a complete Riemannian manifold but $M$ may have other ends besides the AE end $E_0$.
In~\cite{Cecchini_2024po}, the authors use the deformed Dirac operator to prove the classical PMT on manifolds with multiple ends.
Inspired by their results, we deform the weighted Dirac operator used in~\S~\ref{sec:s-end} by adding a suitable potential term and use it to prove the general case.

\subsection{}
We fix a manifold with boundary $M_0$ inside $M$ such that $E_0 \subseteq M_0$ and $M_0 - E_0$ is compact.
We also fix a small positive parameter $\lambda$.
Then, we choose another manifold with boundary $M_{\lambda}$ inside $M$ such that $M_0 \subseteq M_{\lambda}$ and $M_{\lambda} - M_0$ is compact.
Let $\partial (M_{\lambda} - M_0^{\circ}) = \partial_{-} M_{\lambda} \sqcup \partial_+ M_{\lambda}$ with $\partial_- M_{\lambda} = \partial M_0$.
We also require that
\begin{equation*}
  \dist(\partial_-M_{\lambda}, \partial_+ M_{\lambda}) > 1/\lambda.
\end{equation*}

Moreover, we choose a manifold with boundary $M_{\mu}$ inside $M$ again, such that $M_{\lambda} \subseteq M_{\mu}$ and $M_{\mu} - M_{\lambda}$ is compact.
Similarly, let $\partial (M_{\mu} - M_{\lambda}^{\circ}) = \partial_{-} M_{\mu} \sqcup \partial_+ M_{\mu}$ such that $\partial_- M_{\mu} = \partial_+ M_{\lambda}$.
We require that
\begin{equation*}
  \dist(\partial_-M_{\mu}, \partial_+ M_{\mu}) > 3.
\end{equation*}
Note that the completeness of $(M,g)$ ensures the existence of $M_{\lambda}$ and $M_{\mu}$.

The function spaces used in \S~\ref{sec:s-end} can be defined for $\Mm$ in the same way.
In particular, Lemma~\ref{lm:test-f} also holds on $\Mm$ if we further assume that $\supp s \subseteq \Mm^{\circ}$.

However, since $M - E_0$ is noncompact, we need to specify the extended radial coordinate function $r$ used in the weighted norm on $M$.
For our purposes, it is convenient to choose $r$ such that $r \ge 1$ on $M$ and $r \equiv 1$ outside a small neighborhood of $E_0$ (which is contained in $M_0$).
Note that this is always possible by retracting $E_0$ suitably.

\subsection{}
Let $\kappa > 0$ be another parameter.
By~\cite[Lemma~3.3]{Cecchini_2024po}, we can find a smooth function
\begin{equation*}
  h_{\lambda,\kappa}: M_{\lambda}- M_{0} \rightarrow [\lambda, +\infty)
\end{equation*}
such that $h_{\lambda,\kappa} \equiv \lambda$ in a neighborhood of $\partial_{-}M_{\lambda} = \partial M_0$, $h_{\lambda,\kappa} \equiv \kappa$ in a neighborhood of $\partial_{+}M_{\lambda}$, and
\begin{equation}
  \label{eq:h-est}
  h^2_{\lambda,\kappa}(x) - |\nabla h_{\lambda,\kappa}|_x \ge 0
\end{equation}
holds for any $x \in M_{\lambda} - M_0$.

Let $K \subseteq M_0$ be the closure of a collar neighborhood of $\partial M_0$ inside $M_0$ such that $K \cap E_0 = \emptyset$.
Choose a smooth function $\chi: M_0 \rightarrow [0,1]$ such that $\chi|_{M_0 - K} \equiv 0$ and $\chi \equiv 1$ in a neighborhood of $\partial M_0$.

Similar to~\cite{Cecchini_2024po}, we define a smooth function $\varphi_{\lambda,\kappa}: \Mm \rightarrow [0, +\infty)$ as follows,
\begin{equation}
  \label{eq:def-vphi}
  \varphi_{\lambda,\kappa}(x) \coloneqq
  \begin{cases}
    \kappa, & x\in \Mm - M_{\lambda};\\
    h_{\lambda,\kappa}(x), & x\in M_{\lambda} - M_{0};\\
    \lambda \chi(x), & x\in M_0.
  \end{cases}
\end{equation}
By (\ref{eq:h-est}) and (\ref{eq:def-vphi}),
\begin{equation}
  \label{eq:vphi-est}
  \varphi^2_{\lambda,\kappa} - | \nabla \varphi_{\lambda,\kappa} | \ge
  \begin{cases}
    - C_0 \lambda, & x \in K; \\
    \kappa^2, & \Mm- M_{\lambda};\\
    0, & \text{otherwise},
  \end{cases}
\end{equation}
where $C_0 > 0$ is a constant depending only on $\chi$.

Following~\cite{Su_2022no}, we also introduce a pair of cut-off functions on $M_{\mu} - M_{\lambda}$.
Let $d(x)$ be a regularization of the distance function $\dist(x, \partial_{-}M_{\mu})$ on $M_{\mu} - M_{\lambda}$ such that
\begin{equation*}
  | \nabla d |_x \le \frac{3}{2}, \quad x\in M_{\mu} - M_{\lambda}.
\end{equation*}
Choose a smooth function $\Phi: [0, +\infty) \rightarrow [0,1]$ such that $\Phi \equiv 1$ on $[0,1]$, $\Phi \equiv 0$ on $[2, +\infty)$, and $0 \le \Phi' \le 3/2$ on $[0, +\infty)$.
We define a smooth function $\phi: \Mm \rightarrow [0,1]$ by
\begin{equation*}
  \phi(x) \coloneqq
  \begin{cases}
    1, & x\in M_{\lambda};\\
    \Phi({d(x)}), & x \in M_{\mu} - M_{\lambda}.\\
\end{cases}
\end{equation*}
As in~\cite[p.~115]{Bismut_1991aa} and~\cite[(2.14)]{Su_2022no}, let $\phi_{1}, \phi_{2}: \Mm \rightarrow [0,1]$ be
\begin{equation*}
  \begin{gathered}
    \phi_{1} \coloneqq \frac{\phi}{\big(\phi_{\mu}^2+(1-\phi)^2\big)^{\frac{1}{2}}},\quad
    \phi_{2} \coloneqq \frac{1-\phi}{\big(\phi_{\mu}^2+(1-\phi)^2\big)^{\frac{1}{2}}},\\
    \phi^2_{1} + \phi^2_{2} = 1.
  \end{gathered}
\end{equation*}
By the definitions of $d$ and $\phi$, for $i=1,2$, we have
\begin{equation}
  \label{eq:phi12}
  | \nabla \phi_{i} |_x \le
  \begin{cases}
    {C_1},& x\in M_{\mu} - M_{\lambda},\\
    0, & \text{otherwise},
  \end{cases}
\end{equation}
where $C_1$ is a constant independent of the specific choices of $M_{\mu}$ and $\lambda$.

\subsection{}
Let $\sfE = \sfS(\Mm) \oplus \sfS(\Mm)$ be a $\bZ_2$-graded bundle over $\Mm$ such that the Clifford action $\caE(-)$ on $\sfE$ is an odd operator, that is, for $v_x\in \rT_x \Mm$ and $\psi_{i,x} \in \sfS_x(\Mm)$,
\begin{equation*}
  \caE(v_x) (\psi_{1,x}, \psi_{2,x}) \coloneqq (\opc(v_x) \psi_{2,x}, \opc(v_x) \psi_{1,x}).
\end{equation*}
As a result, the (twisted) Dirac operator on $\DE$ is
\begin{equation*}
  \DE \coloneqq \sum_{i=1}^n \caE(e_i) \nabla_{e_i} =
  \begin{bmatrix}
    0 & \opD \\
    \opD & 0
  \end{bmatrix}
  : \Gamma(\Mm, \sfE) \rightarrow \Gamma(\Mm, \sfE),
\end{equation*}
where $e_1,\cdots,e_n$ is a locally othornormal basis for $\rT M_{\mu}$.
Moreover, the weighted Dirac operator on $\DE$ is
\begin{equation*}
  \wDE \coloneqq \rho^{-1/2} \DE \rho^{1/2} = \DE + \frac{1}{2}\caE(\nabla \log \rho) =
  \begin{bmatrix}
    0 & \wD \\
    \wD & 0
  \end{bmatrix}.
\end{equation*}

Note that there exists another natural odd endomorphism on $\sfE$,
\begin{equation*}
  \opV \coloneqq
  \begin{bmatrix}
    0 & -i \\
    i & 0
  \end{bmatrix}.
\end{equation*}
By definition,
\begin{equation}
  \label{eq:dv-comm}
  [\DE, \opV]_+ \coloneqq \DE \opV + \opV \DE = 0 \text{ and } [\wDE, \opV]_+ = \rho^{-1/2}[\DE, \opV]_+\rho^{1/2} = 0.
\end{equation}

As in~\cite{Cecchini_2024po}, in contrast to the single end case, we now use the deformed Dirac operator
\begin{equation*}
  \dwDE \coloneqq \wDE + \varphi_{\lambda,\kappa} \opV.
\end{equation*}

\subsection{A local boundary condition.}
Let $\opP \coloneq \caE(\nu) \opV: \sfE|_{\partial M_{\mu}} \rightarrow \sfE|_{\partial M_{\mu}}$ be a bundle endomorphism defined on $\partial M_{\mu}$, where $\nu$ is the outward-pointing unit normal vector field on $\partial \Mm$.
Then, let
\begin{gather*}
  \Gamma_c(\Mm, \sfE, \opP) \coloneqq \aset{s\in \Gamma_c(\Mm, \sfE)| \opP(s|_{\partial \Mm}) = -s|_{\partial \Mm}},\\
  \rW^{1,2}_{-(n-2)/2}(\sfE,\rho \dv, \opP) \coloneqq \aset{s\in \rW^{1,2}_{-(n-2)/2}(\sfE,\rho \dv)| \opP(s|_{\partial \Mm}) = -s|_{\partial \Mm}}.
\end{gather*}
Note that $\Gamma_c(\Mm, \sfE, \opP)$ is dense in $\rW^{1,2}_{-(n-2)/2}(\sfE,\rho \dv, \opP)$.
As in (\ref{eq:bd-term}), for $s\in \Gamma_c(\Mm, \sfE, \opP)$,
\begin{equation}
  \label{eq:bd-mm}
  \begin{aligned}
    & \int_{\Mm} \big(|\nabla s|^2 - |\wDE s|^2 + \frac{1}{4}  (\scv - 2 \Delta \log \rho - |\nabla \log \rho|^2) |s|^2 \big)\rho\dv\\
    = & \int_{\partial \Mm} \langle s, \nabla_{\nu}s + \caE(\nu) \wDE s \rangle \rho\dA \\
    = & \int_{\partial \Mm} \langle s, \nabla_{\nu}s + \caE(\nu) (\caE(\nu)\nabla_{\nu} - \caE(\nu) \opD^{\partial,\sfE} + \caE(\nu)\frac{H}{2} + \frac{1}{2}\caE( \nabla \log \rho)) s \rangle \rho\dA \\
    = & \int_{\partial \Mm} \Big(\langle s, \opD^{\partial,\sfE}_\rho s\rangle - \frac{H + \nabla_{\nu} \log \rho}{2} |s|^2 \Big) \rho\dA
        = -\int_{\partial \Mm} \frac{H + \nabla_{\nu} \log \rho}{2} |s|^2 \rho\dA,
  \end{aligned}
\end{equation}
where $\opD^{\partial,\sfE}$ is the boundary operator of $\opD^{\sfE}$, $\opD^{\partial,\sfE}_{\rho} = \rho^{-1/2} \opD^{\partial,\sfE} \rho^{1/2}$ and $H$ is the mean curvature of $\partial \Mm$.
Note that the last equality of (\ref{eq:bd-mm}) holds because $\opD^{\partial,\sfE}_{\rho} \opP = - \opP \opD^{\partial,\sfE}_{\rho}$.

At the same time, by the divergence theorem, for $s\in \Gamma_c(\Mm, \sfE, \opP)$,
\begin{multline}
  \label{eq:bd-mm2}
  \int_{\Mm} \langle\wDE(s) , \varphi_{\lambda,\kappa} \opV ( s)\rangle \rho \dv - \int_{\Mm} \langle s , \wDE(\varphi_{\lambda,\kappa}\opV ( s))\rangle \rho \dv \\
  = - \int_{\partial \Mm} \langle s, \caE(\nu)(\varphi_{\lambda,\kappa} \opV ( s))\rangle \rho \dv
  = \kappa^2 \int_{\partial \Mm} |s|^2 \rho \dv.
\end{multline}
Combining (\ref{eq:bd-mm}) and (\ref{eq:bd-mm2}), for $s\in \Gamma_c(\Mm, \sfE, \opP)$, we obtain
\begin{multline}
  \label{eq:dwde-sqrt}
  \int_{\Mm} |\dwDE s|^2 \rho \dv = \int_{\Mm} \big(|\nabla s|^2 + \frac{1}{4}  (\scv - 2 \Delta \log \rho - |\nabla \log \rho|^2) |s|^2 \big)\rho\dv \\
  + \int_{\Mm} \langle s , \caE(\nabla \varphi_{\lambda,\kappa})\opV  ( s)\rangle \rho \dv
  + \int_{\Mm} \varphi_{\lambda,\kappa}^2|s|^2 \rho \dv + \int_{\partial \Mm} \frac{H + \nabla_{\nu} \log \rho + 2\kappa^2}{2} |s|^2 \rho\dA.
\end{multline}
By a density argument, (\ref{eq:dwde-sqrt}) also holds for $s\in \rW^{1,2}_{-(n-2)/2}(\sfE,\rho \dv, \opP)$.

\subsection{}
As in the single end case, to prove Theorem~\ref{thm:bw} in the multiple ends case, we first need to show that $\dwDE$ is invertible in some sense.
Since $\Mm$ is a manifold with boundary, the local boundary condition introduced in the previous subsection is required.

By~\cite[Theorem~2.12]{Cecchini_2024po},
\begin{equation}
  \label{eq:wd-isom-E}
  \dwDE: \rW^{1,2}_{-(n-2)/2}(\sfE,\rho \dv, \opP) \rightarrow \rL^2(\sfE, \rho \dv)
\end{equation}
is a Fredholm operator with nonnegative index.\footnote{Strictly speaking,~\cite[Theorem~2.12]{Cecchini_2024po} proves the $\rho = 1$ case.
However, the same proof also works for the general case.}
Therefore, to show (\ref{eq:wd-isom-E}) is an isomorphism, we only need to show that it is injective for suitably chosen $\lambda,\kappa$.

For $s\in \rW^{1,2}_{-(n-2)/2}(\sfE,\rho \dv, \opP)$, using the cut-off functions $\phi_{1}$ and $\phi_{2}$, we have\footnote{In this section, unless explicitly stated otherwise, the norm $\norm{\bullet}$ is evaluated on $\Mm$.}
\begin{equation}
  \label{eq:dd-s}
  \int_{\Mm} |\dwDE s|^2 \rho \dv = \norm{\dwDE s}^2_{\rL^2(\rho)} = \norm{\phi_{1} \dwDE s}^2_{\rL^2(\rho)} + \norm{\phi_{2} \dwDE s}^2_{\rL^2(\rho)}.
\end{equation}
For $i = 1, 2$, by (\ref{eq:phi12}),
\begin{multline}
  \label{eq:dd-s1}
  \norm{\phi_{i} \dwDE s}^2_{\rL^2(\rho)} = \norm{\dwDE( \phi_{i}s) - \caE(\nabla \phi_{i}) s}^2_{\rL^2(\rho)} \\
  \ge \frac{1}{2} \norm{\dwDE (\phi_{i}s)}^2_{\rL^2(\rho)} - \norm{\caE(\nabla \phi_{i}) s}^2_{\rL^2(\rho)}
  \ge \frac{1}{2} \norm{\dwDE (\phi_{i}s)}^2_{\rL^2(\rho)} - {C_1^2}\int_{M_{\mu} - M_{\lambda}} |s|^2 \rho \dv.
\end{multline}

Since $\supp (\phi_{2}s) \subseteq N - M_{\lambda}$, by (\ref{eq:def-vphi}), $\varphi_{\lambda,\kappa}(\phi_{2}s) = \kappa \phi_{2}s$.
Therefore, (\ref{eq:dwde-sqrt}) implies that
\begin{equation}
  \label{eq:d-phi2}
  \norm{\dwDE (\phi_{2}s)}^2_{\rL^2(\rho)} \ge \norm{\nabla (\phi_{2}s)}^2_{\rL^2(\rho)} + (\kappa^2 - C(g_{\Mm, \rho})) \norm{ \phi_{2}s}^2_{\rL^2(\rho)}  + (\kappa^2 - C(g_{\Mm, \rho}))\int_{\partial \Mm} |s|^2 \rho\dA,
\end{equation}
where $C(g_{\Mm}, \rho)$ is a positive constant depending on $g_{\Mm}$ and $\rho|_{\Mm}$.

Regarding $\phi_{1}s$, we first note that $\supp (\phi_{1}s) \subseteq M_{\mu}^\circ$.
Hence, by (\ref{eq:vphi-est}) and (\ref{eq:dwde-sqrt}),
\begin{equation}
  \label{eq:dd-phi1}
  \begin{aligned}
    & \norm{\dwDE (\phi_{1}s)}^2_{\rL^2(\rho)}\\
    = & \norm{\wDE (\phi_{1}s)}^2_{\rL^2(\rho)}
    + \int_{\Mm} \langle\phi_{1}s , \caE(\nabla \varphi_{\lambda,\kappa})\opV  (\phi_{1} s)\rangle \rho \dv
    + \int_{\Mm} \varphi_{\lambda,\kappa}^2|\phi_{1}s|^2 \rho \dv \\
    \ge & \norm{\wDE (\phi_{1}s)}^2_{\rL^2(\rho)} + \int_{\Mm} (\varphi_{\lambda,\kappa}^2 - |\nabla \varphi_{\lambda,\kappa}|) |\phi_{1} s|^2 \rho \dv \\
    \ge & \norm{\wDE (\phi_{1}s)}^2_{\rL^2(\rho)} + \kappa^2 \int_{M_{\mu} - M_{\lambda}} |\phi_{1} s|^2 \rho \dv - C_0 \lambda \int_{K} |s|^2 \rho \dv,
\end{aligned}
\end{equation}
where we also use $\phi_{1}|_K \equiv 1$.

Now, since $\supp (\phi_{1}s) \subseteq M_{\mu}^{\circ}$, $\phi_{1}s$ can be naturally viewed as a section over $M$.
Meanwhile, (\ref{eq:pt-kato-s}) is a pointwise estimate for sections defined on $M$.
As a result, (\ref{eq:pt-kato-s}) also holds\footnote{By definition, the estimate (\ref{eq:pt-kato-s}) holds for sections of $\sfS(M)$.
However, since $\sfE|_{M_{\mu}} = \sfS(M_{\mu}) \oplus \sfS(M_{\mu})$, (\ref{eq:pt-kato-s}) also applies to sections of $\sfE|_{M_{\mu}}$ used here.}
for $\phi_{1}s$.
Moreover, note that the right hand side of (\ref{eq:bd-mm}) vanishes for $\phi_{1}s$.
Thus, by (\ref{eq:pt-kato-s}) and (\ref{eq:bd-mm}), we can apply a variant of (\ref{eq:est-wD}) to $\phi_{1}s$,
\begin{multline*}
  (1 + \frac{1}{(n-1)l})\int_{M_{\mu}} |\wDE (\phi_{1} s)|^2 \rho\dv \ge \frac{c_0}{2} \int_{M_{\mu}} |\nabla |\phi_{1} s||^2 \rho \dv \\
  + \frac{1}{2} \Big(\int_{M_{\mu}} |\nabla |\phi_{1} s||^2 \rho \dv
  + \frac{1}{2} \int_{M_{\mu}} (\scv - 2 \Delta \log \rho - \frac{n+1}{n+2} | \nabla \log \rho|^2 ) |\phi_{1} s|^2 \rho\dv \Big).
\end{multline*}
Since $\supp (\phi_{1}s) \subseteq M_{\mu}^{\circ}$, by (\ref{eq:bw-pos}) and Lemma~\ref{lm:test-f}, the term in parentheses on the right hand side of the above formula is nonnegative.
In other words, we can find a constant $C_2 > 0$, independent of $\lambda, \kappa$, such that
\begin{equation}
  \label{eq:d-phi1}
  \norm{\nabla|\phi_{1}s|}^2_{\rL^2(\rho)} \le  C_2\norm{\wDE (\phi_{1}s)}^2_{\rL^2(\rho)},\quad C_2 = \frac{2}{c_0} + \frac{2}{(n-1)lc_0} .
\end{equation}

By (\ref{eq:dd-s})-(\ref{eq:d-phi1}), we have
\begin{multline}
  \label{eq:dwde-est0}
  \norm{\dwDE s}^2_{\rL^2(\rho)} \ge \frac{1}{2C_2}\norm{\nabla|\phi_{1}s|}^2_{\rL^2(\rho)} + \frac{1}{2}\norm{\nabla (\phi_{2}s)}^2_{\rL^2(\rho)}\\
  + \frac{\kappa^2}{2} \int_{M_{\mu} - M_{\lambda}} |\phi_{1} s|^2 \rho \dv - \frac{C_0 \lambda}{2} \int_{K} |s|^2 \rho \dv
  - {2C_1^2}\int_{M_{\mu} - M_{\lambda}} |s|^2 \rho \dv \\
+ \frac{\kappa^2 - C(g_{\Mm, \rho})}{2} \norm{ \phi_{2}s}^2_{\rL^2(\rho)}  + \frac{\kappa^2 - C(g_{\Mm, \rho})}{2}\int_{\partial \Mm} |s|^2 \rho\dA.
\end{multline}

On $M_0$, we need the weighted Poincar\'e inequality,
\begin{equation}
  \label{eq:bd-wp-ineq}
  \norm{s}^2_{\rL^2_{-(n-2)/2}(M_0, \rho\dv)} \le C_3 \norm{\nabla|s|}^2_{\rL^2(M_0,\rho\dv)},
\end{equation}
where $C_3$ is a constant independent of $\lambda,\kappa$.
As explained in~\cite[Proposition~2.7]{Cecchini_2024po}, the above inequality can be deduced from the usual weighted Poincar\'e inequality by a doubling argument.

Since $K \cap E_0 = \emptyset$, by our choice of $r$ on $M$, we have
\begin{equation}
  \label{eq:K-est}
  \int_{K} |s|^2 \rho \dv \le \norm{s}^2_{\rL^2_{-(n-2)/2}(M_0, \rho\dv)}.
\end{equation}

Now, we fix $\lambda$ such that
\begin{equation}
  \label{eq:lambda-v}
  \lambda^{-1} \ge 2C_0C_2C_3.
\end{equation}
By the above inequality, (\ref{eq:dwde-est0}), (\ref{eq:bd-wp-ineq}) and (\ref{eq:K-est}),
\begin{multline}
  \label{eq:dwde-est1}
  \norm{\dwDE s}^2_{\rL^2(\rho)} \ge \frac{1}{4C_2}\norm{\nabla|\phi_{1}s|}^2_{\rL^2(\rho)} + \frac{1}{2}\norm{\nabla (\phi_{2}s)}^2_{\rL^2(\rho)}\\
  + \frac{\kappa^2}{2} \int_{M_{\mu} - M_{\lambda}} |\phi_{1} s|^2 \rho \dv
  - {2C_1^2}\int_{M_{\mu} - M_{\lambda}} |s|^2 \rho \dv \\
+ \frac{\kappa^2 - C(g_{\Mm, \rho})}{2} \norm{ \phi_{2}s}^2_{\rL^2(\rho)}  + \frac{\kappa^2 - C(g_{\Mm, \rho})}{2}\int_{\partial \Mm} |s|^2 \rho\dA.
\end{multline}

Next, we fix a choice of $\Mm$ and choose
\begin{equation}
  \label{eq:kappa-v}
  \kappa \ge \sqrt{C(g_{\Mm}, \rho) + 4C^2_1 + 2}.
\end{equation}
Then,
\begin{multline*}
  \frac{\kappa^2}{2} \int_{M_{\mu} - M_{\lambda}} |\phi_{1} s|^2 \rho \dv
  - {2C_1^2}\int_{M_{\mu} - M_{\lambda}} |s|^2 \rho \dv \\
  + \frac{\kappa^2 - C(g_{\Mm, \rho})}{2} \norm{ \phi_{2}s}^2_{\rL^2(\rho)}
  + \frac{\kappa^2 - C(g_{\Mm, \rho})}{2}\int_{\partial \Mm} |s|^2 \rho\dA \\
  \ge (2C_1^2 + 1) \Big(\int_{M_{\mu} - M_{\lambda}} |\phi_{1} s|^2 \rho \dv + \int_{M_{\mu} - M_{\lambda}} |\phi_{2} s|^2 \rho \dv\Big)
  -  {2C_1^2}\int_{M_{\mu} - M_{\lambda}} |s|^2 \rho \dv \\
  = \int_{M_{\mu} - M_{\lambda}} |s|^2 \rho \dv.
\end{multline*}
Plugging the above inequality into (\ref{eq:dwde-est1}), we get
\begin{equation}
  \label{eq:dwde-est11}
  \norm{\dwDE s}^2_{\rL^2(\rho)} \ge \frac{1}{4C_2}\norm{\nabla|\phi_{1}s|}^2_{\rL^2(\rho)} + \frac{1}{2}\norm{\nabla (\phi_{2}s)}^2_{\rL^2(\rho)} + \int_{M_{\mu} - M_{\lambda}} |s|^2 \rho \dv.
\end{equation}

If $\dwDE s = 0$, by (\ref{eq:dwde-est11}), we know that $s|_{\Mm - M_{\lambda}} \equiv 0$ and $|\phi_1 s|$ and $|\phi_2s|$ are constant, which implies that $s = 0$.

In summary, we have shown the following result.
First, we choose $\lambda$ in the range (\ref{eq:lambda-v}).
Second, we fix a manifold with boundary $\Mm$ based on $\lambda$.
Third, we choose $\kappa$ in the range (\ref{eq:kappa-v}) based on $\Mm$ and $\rho|_{\Mm}$.
Then, (\ref{eq:wd-isom-E}) must be an isomorphism.

For later application, we also notice that for such choices of $\lambda,\kappa$, by (\ref{eq:dwde-est0}),  (\ref{eq:bd-wp-ineq}) and (\ref{eq:K-est}), the following estimate also holds,
\begin{multline}
  \label{eq:dwde-est2}
  \norm{\dwDE s}^2_{\rL^2(\rho)} \ge \frac{1}{4C_2}\norm{\nabla|\phi_{1}s|}^2_{\rL^2(\rho)} \ge \frac{1}{4C_2}\int_{M_0}|\nabla|s||^2 \rho \dv \\
  \ge \frac{1}{4C_2C_3}\norm{s}^2_{\rL^2_{-(n-2)/2}(M_0, \rho\dv)} \ge \frac{1}{4C_2C_3}\int_{K} |s|^2 \rho \dv.
\end{multline}

\subsection{}
As in the single end case, the isomorphism (\ref{eq:wd-isom-E}) enables us to choose a Witten-type spinor on $\Mm$.
First, let $\Psi^{\sfE}_{\infty} \coloneqq (\Psi_{\infty}, \Psi_{\infty})$ and require that
\begin{equation*}
  \supp \Psi^{\sfE}_{\infty} \subseteq M_0 \text{ and } \supp \Psi^{\sfE}_{\infty} \cap K = \emptyset.
\end{equation*}
Then, let $\xi \coloneqq \dwDE(\Psi^{\sfE}_{\infty}) = \wDE(\Psi^{\sfE}_{\infty}) \in \rL^2(\sfE, \rho \dv)$.
Having fixed $\lambda,\kappa$ as explained above, the isomorphism (\ref{eq:wd-isom-E}) guarantees the existence of $\eta \in \rW^{1,2}_{-(n-2)/2}(\sfE,\rho \dv,\allowbreak \opP)$ such that $\dwDE(\eta) = -\xi$.
As a result, taking $\psi = \Psi^{\sfE}_{\infty} + \eta$, we have
\begin{equation}
  \label{eq:dwde-ws}
  \dwDE \psi = 0, \quad \psi - \Psi^{\sfE}_{\infty} \in \rW^{1,2}_{-(n-2)/2}(\sfE,\rho \dv, \opP).
\end{equation}

By (\ref{eq:a-b}) and (\ref{eq:bd-term}), we have
\begin{multline*}
  (n-2)\omega_{n-1}((n-1)\alpha + 2\beta) + \int_{\Mm} |\wDE(\phi_1\psi)^2| \rho \dv \\
  = \int_{\Mm} \big(|\nabla (\phi_1 \psi)|^2 + \frac{1}{4}  (\scv - 2 \Delta \log \rho - |\nabla \log \rho|^2) |\phi_1 \psi|^2 \big)\rho\dv.
\end{multline*}
Using (\ref{eq:pt-kato-s}), the above equality gives
\begin{multline*}
  (n-2)\omega_{n-1}((n-1)\alpha + 2\beta) + (1 + \frac{1}{(n-1)l})\int_{M_{\mu}} |\wDE (\phi_{1} \psi)|^2 \rho\dv \ge \frac{c_0}{2} \int_{M_{\mu}} |\nabla |\phi_{1} \psi||^2 \rho \dv \\
  + \frac{1}{2} \Big(\int_{M_{\mu}} |\nabla |\phi_{1} \psi||^2 \rho \dv
  + \frac{1}{2} \int_{M_{\mu}} (\scv - 2 \Delta \log \rho - \frac{n+1}{n+2} | \nabla \log \rho|^2 ) |\phi_{1} \psi|^2 \rho\dv \Big).
\end{multline*}
Applying (\ref{eq:bw-pos}) and Lemma~\ref{lm:test-f} to $|\phi_{1} \psi|$, the above inequality implies
\begin{equation*}
  \frac{2(n-1)\omega_{n-1}}{c_0C_2}((n-1)\alpha + 2\beta) + \int_{M_{\mu}} |\wDE (\phi_{1} \psi)|^2 \rho\dv \ge C_2^{-1}\int_{M_{\mu}} |\nabla |\phi_{1} \psi||^2 \rho \dv + \frac{1}{c_0C_2}\opQ(|\phi_{1}\psi|),
\end{equation*}
where $\opQ(|\phi_{1} \psi|) = \int_{\Mm}Q|\phi_{1}\psi|^2 \rho \dv$.

Note that $\supp \varphi_{\lambda,\kappa} \phi_{1} \psi \subseteq \Mm^{\circ} \cap K$, so the divergence theorem yields,
\begin{equation*}
  \int_{\Mm} \langle\wDE(s) , \varphi_{\lambda,\kappa} \opV ( \phi_{1}s)\rangle \rho \dv = \int_{\Mm} \langle s , \wDE(\varphi_{\lambda,\kappa}\opV ( \phi_{1}s))\rangle \rho \dv.
\end{equation*}
By the above two formulas and (\ref{eq:vphi-est}), the following inequality, similar to (\ref{eq:dd-phi1}), holds,
\begin{equation*}
  \begin{aligned}
    & \frac{2(n-1)\omega_{n-1}}{c_0C_2}((n-1)\alpha + 2\beta) + \int_{M_{\mu}} |\dwDE (\phi_{1} \psi)|^2 \rho\dv \\
    \ge & C_2^{-1}\int_{M_{\mu}} |\nabla |\phi_{1} \psi||^2 \rho \dv + \int_{\Mm} (\varphi_{\lambda,\kappa}^2 - |\nabla \varphi_{\lambda,\kappa}|) |\phi_{1} s|^2 \rho \dv + \frac{1}{c_0C_2}\opQ(|\phi_{1}\psi|) \\
    \ge & C_2^{-1}\int_{M_{\mu}} |\nabla |\phi_{1} \psi||^2 \rho \dv + \kappa^2 \int_{M_{\mu} - M_{\lambda}} |\phi_{1} s|^2 \rho \dv - C_0 \lambda \int_{K} |s|^2 \rho \dv + \frac{1}{c_0C_2}\opQ(|\phi_{1}\psi|).
  \end{aligned}
\end{equation*}

Meanwhile, we remark that (\ref{eq:dd-s}) and (\ref{eq:dd-s1}) also hold for $\psi$.
As a result, combining (\ref{eq:dd-s}), (\ref{eq:dd-s1}), (\ref{eq:dwde-ws}) and the above inequality, we obtain
\begin{equation*}
  \begin{aligned}
    & \frac{(n-1)\omega_{n-1}}{c_0C_2}((n-1)\alpha + 2\beta) \\
    = & \frac{(n-1)\omega_{n-1}}{c_0C_2}((n-1)\alpha + 2\beta) + \int_{M_{\mu}} |\dwDE (\psi)|^2 \rho\dv \\
    \ge & \frac{(n-1)\omega_{n-1}}{c_0C_2}((n-1)\alpha + 2\beta) + \frac{1}{2} \norm{\dwDE (\phi_{1}\psi)}^2_{\rL^2(\rho)} \\
    & \hspace{4cm}+ \frac{1}{2} \norm{\dwDE (\phi_{2}\psi)}^2_{\rL^2(\rho)} - {2C_1^2}\int_{M_{\mu} - M_{\lambda}} |\psi|^2 \rho \dv \\
    \ge & \frac{1}{2C_2}\int_{M_{\mu}} |\nabla |\phi_{1} \psi||^2 \rho \dv + \frac{\kappa^2}{2} \int_{M_{\mu} - M_{\lambda}} |\phi_{1} \psi|^2 \rho \dv - \frac{C_0 \lambda}{2} \int_{K} |\psi|^2 \rho \dv \\
    & \hspace{3cm}+ \frac{1}{2c_0C_2}\opQ(|\phi_{1}\psi|) + \frac{1}{2} \norm{\dwDE (\phi_{2}\psi)}^2_{\rL^2(\rho)} - {2C_1^2}\int_{M_{\mu} - M_{\lambda}} |\psi|^2 \rho \dv.
  \end{aligned}
\end{equation*}
Then, plugging (\ref{eq:d-phi2}) into the above inequality, we have an estimate analogous to (\ref{eq:dwde-est0})
\begin{multline}
  \label{eq:mass-est0}
  \frac{(n-1)\omega_{n-1}}{c_0C_2}((n-1)\alpha + 2\beta) \ge \frac{1}{2C_2}\norm{\nabla |\phi_{1}\psi|}^2_{\rL^2(\rho)} + \frac{1}{2}\norm{\nabla (\phi_{2}\psi)}^2_{\rL^2(\rho)}\\
  + \frac{\kappa^2}{2} \int_{M_{\mu} - M_{\lambda}} |\phi_{1} \psi|^2 \rho \dv
  - \frac{C_0 \lambda}{2} \int_{K} |\psi|^2 \rho \dv
  - {2C_1^2}\int_{M_{\mu} - M_{\lambda}} |\psi|^2 \rho \dv \\
  + \frac{1}{2c_0C_2}\opQ(|\phi_{1}\psi|)
  + \frac{\kappa^2 - C(g_{\Mm, \rho})}{2} \norm{ \phi_{2}\psi}^2_{\rL^2(\rho)}
  + \frac{\kappa^2 - C(g_{\Mm, \rho})}{2}\int_{\partial \Mm} |\psi|^2 \rho\dA.
\end{multline}

Unlike in (\ref{eq:dwde-est0}), since $|\psi| \notin \rW^{1,2}_{-(n-2)/2}$, we cannot apply the weighted Poincar\'e inequality directly to control the integral over $K$ in (\ref{eq:mass-est0}).
Here, we use an argument in~\cite{Cecchini_2024po}.
Since $\Psi^{\sfE}_{\infty}|_{K} \equiv 0$, by definition, $\psi|_K = \eta|_K$.
Due to $\eta \in \rW^{1,2}_{-(n-2)/2}(\sfE,\rho \dv,\allowbreak \opP)$, by (\ref{eq:dwde-sqrt}),
\begin{multline*}
  \int_{K} |\psi|^2 \rho \dv = \int_{K} |\eta|^2 \rho \dv \le 4C_2 C_3\norm{\dwDE \eta}^2_{\rL^2(\rho)} \\
  = 4C_2 C_3 \norm{\dwDE \Psi^{\sfE}_{\infty}}^2_{\rL^2(M_0,\rho\dv)} = 4C_2C_3 \norm{\wDE \Psi^{\sfE}_{\infty}}^2_{_{\rL^2(M_0,\rho\dv)}}.
\end{multline*}
Now, by choosing $\lambda$ and $\kappa$ as in (\ref{eq:lambda-v}) and (\ref{eq:kappa-v}), the above inequality and (\ref{eq:mass-est0}) give
\begin{equation}
  \label{eq:est-ab}
  \frac{(n-1)\omega_{n-1}}{c_0C_2}((n-1)\alpha + 2\beta) \ge - 4C_0C_2C_3 \lambda \norm{\wDE \Psi^{\sfE}_{\infty}}^2_{\rL^2(M_0,\rho\dv)} + \frac{1}{2c_0C_2}\opQ(|\phi_{1}\psi|).
\end{equation}
Note that $\|\wDE \Psi^{\sfE}_{\infty}\|^2_{\rL^2(M_0,\rho\dv)}$ is a constant and $\lambda$ in the range (\ref{eq:lambda-v}) can be arbitrarily small.
If the scalar curvature of $g$ is nonnegative in the BW sense, the above estimate implies that $(n-1)\alpha + 2\beta \ge 0$.

If the scalar curvature of $g$ is positive in the BW sense, to show $(n-1)\alpha + 2\beta > 0$, we need a few more arguments.
Let $\aset{\lambda_j}$ be a sequence such that $\lambda_j \rightarrow 0$ as $j \rightarrow +\infty$ and let $\psi_j$ be the Witten-type spinor defined in (\ref{eq:dwde-ws}) with respect to $\lambda_j$ (and corresponding $\kappa_j$).
By definition, we can find
\begin{equation*}
  \eta_j \in \rW^{1,2}_{-(n-2)/2}(\sfE,\rho \dv, \opP),\quad \dwDE(\eta_j) = -\xi = - \wDE(\Psi^{\sfE}_{\infty}),\quad \psi_j = \eta_j + \Psi^{\sfE}_{\infty},
\end{equation*}
where $\Mm$ also depends on $\lambda_j$.

By (\ref{eq:dwde-est0}) and our choice of $\lambda,\kappa$, we have
\begin{equation*}
  \norm{\wDE(\Psi^{\sfE}_{\infty})}^2_{\rL^2(M_0,\rho\dv)} = \norm{\dwDE \eta_j}^2_{\rL^2(\rho)}  \ge \frac{1}{4C_2}\int_{M_0}|\nabla|\eta_j||^2 \rho \dv.
\end{equation*}
Then, combining the above estimate with (\ref{eq:bd-wp-ineq}),
\begin{equation*}
  4C_2(C_3 + 1)\norm{\wDE(\Psi^{\sfE}_{\infty})}^2_{\rL^2(M_0,\rho\dv)} \ge \norm{\eta_j}^2_{\rW^{1,2}_{-(n-2)/2}(M_0, \rho\dv)}.
\end{equation*}
Since the left hand side of the above estimate is a constant, by the weighted Rellich–Kondrachov compact embedding theorem,~\cite[Proposition~2.8]{Cecchini_2024po} or~\cite[Theorem~A.26]{Lee_2019ge}, we can find a subsequence of $\aset{\eta_j}$ converging in $\rL^2_{-(n-3)/2}(M_0, \rho\dv)$.
Without loss of generality, we assume that there exists $\eta_0\in \rL^2_{-(n-3)/2}(M_0, \rho\dv)$ such that
\begin{equation*}
  \lim_{j\rightarrow \infty} \norm{\eta_j - \eta_0}_{\rL^2_{-(n-3)/2}(M_0, \rho\dv)} = 0.
\end{equation*}
For $f\in \rL^2_{-(n-3)/2}(M_0, \rho\dv)$, we define a functional $\opB$ on $\rL^2_{-(n-3)/2}(M_0, \rho\dv)$ as follows,
\begin{equation*}
  \opB(f) \coloneqq \int_{M_0} Q |\Psi^{\sfE}_{\infty} + f|^2 \rho \dv.
\end{equation*}
By the definition of $Q$, (\ref{eq:def-Q}), $\opB$ is a continuous functional on $\rL^2_{-(n-3)/2}(M_0, \rho\dv)$.
As a result,
\begin{equation*}
  \opB(\eta_0) = \int_{M_0} Q |\Psi^{\sfE}_{\infty} + \eta_0|^2 \rho \dv = \lim_{j\rightarrow \infty} \opB(\eta_j).
\end{equation*}
Since $\Psi^{\sfE}_{\infty} \notin \rL^2_{-(n-3)/2}(M_0, \rho\dv)$, $\Psi^{\sfE}_{\infty} + \eta_0$ must be non-zero; thus $\opB(\eta_0) > 0$.
Then, by (\ref{eq:est-ab}),
\begin{multline*}
  \frac{(n-1)\omega_{n-1}}{c_0C_2}((n-1)\alpha + 2\beta) \ge \lim_{j \rightarrow \infty}- 4C_0C_2C_3 \lambda_j \norm{\wDE \Psi^{\sfE}_{\infty}}^2_{\rL^2(M_0,\rho\dv)} + \overline{\lim_{j \rightarrow \infty}}\frac{1}{2c_0C_2}\opQ(|\phi_{1}\psi_j|)\\
  \ge \frac{1}{2c_0C_2} \lim_{j \rightarrow \infty} \opB(\eta_j) = \frac{1}{2c_0C_2}\opB(\eta_0) > 0.
\end{multline*}
The proof of Theorem~\ref{thm:bw} is complete.

\begin{remark}
  \label{rk:def-Q}
As is clear from the proof of Theorem~\ref{thm:bw}, to establish $(n-1)\alpha + 2\beta > 0$, it in fact suffices to assume that $Q \ge 0$ and that $Q$ satisfies (\ref{eq:def-Q}).
  Notably, this condition is also sufficient for the proof given in~\cite{Brendle_2026dh}.
\end{remark}

\bibliographystyle{amsplain}

\begin{bibdiv}
\begin{biblist}

\bib{Bei_2026ge}{article}{
      author={Bei, Francesco},
      author={Cecchini, Simone},
       title={Geometric rigidity via harmonic twisted spinors},
        date={2026},
      eprint={2606.19567v1},
         url={https://arxiv.org/abs/2606.19567v1},
}

\bib{Bi_2026pr}{article}{
      author={Bi, Yuchen},
      author={Hao, Tianze},
      author={He, Shihang},
      author={Shi, Yuguang},
      author={Zhu, Jintian},
       title={A proof for the {R}iemannian positive mass theorem up to
  dimension 19},
        date={2026},
      eprint={2603.02769v2},
         url={https://arxiv.org/abs/2603.02769v2},
}

\bib{Bismut_1991aa}{article}{
      author={Bismut, Jean-Michel},
      author={Lebeau, Gilles},
       title={Complex immersions and {Q}uillen metrics},
        date={1991},
        ISSN={0073-8301},
     journal={Inst. Hautes {\'E}tudes Sci. Publ. Math.},
      number={74},
       pages={ii+298 pp. (1992)},
         url={http://mathscinet.ams.org/mathscinet-getitem?mr=1188532},
      review={\MR{1188532}},
}

\bib{Balfauf_2022sm}{article}{
      author={Baldauf, Julius},
      author={Ozuch, Tristan},
       title={Spinors and mass on weighted manifolds},
        date={2022},
        ISSN={0010-3616},
     journal={Comm. Math. Phys.},
      volume={394},
      number={3},
       pages={1153\ndash 1172},
         url={https://mathscinet.ams.org/mathscinet-getitem?mr=4470247},
      review={\MR{4470247}},
}

\bib{Brendle_2026dh}{article}{
      author={Brendle, S.},
      author={Wang, Y.},
       title={A dimension descent scheme for the positive mass theorem in
  arbitrary dimension},
        date={2026},
      eprint={2604.08473v2},
         url={http://arxiv.org/abs/2604.08473v2},
}

\bib{Calderbank_2000re}{article}{
      author={Calderbank, David M.~J.},
      author={Gauduchon, Paul},
      author={Herzlich, Marc},
       title={Refined {K}ato inequalities and conformal weights in {R}iemannian
  geometry},
        date={2000},
        ISSN={0022-1236,1096-0783},
     journal={J. Funct. Anal.},
      volume={173},
      number={1},
       pages={214\ndash 255},
         url={https://mathscinet.ams.org/mathscinet-getitem?mr=1760284},
      review={\MR{1760284}},
}

\bib{Cecchini_2024po}{article}{
      author={Cecchini, Simone},
      author={Zeidler, Rudolf},
       title={The positive mass theorem and distance estimates in the spin
  setting},
        date={2024},
        ISSN={0002-9947,1088-6850},
     journal={Trans. Amer. Math. Soc.},
      volume={377},
      number={8},
       pages={5271\ndash 5288},
         url={https://mathscinet.ams.org/mathscinet-getitem?mr=4771222},
      review={\MR{4771222}},
}

\bib{Chu_2024no}{article}{
      author={Chu, Jianchun},
      author={Zhu, Jintian},
       title={A non-spin method to the positive weighted mass theorem for
  weighted manifolds},
        date={2024},
        ISSN={1050-6926,1559-002X},
     journal={J. Geom. Anal.},
      volume={34},
      number={9},
       pages={Paper No. 272, 31},
         url={https://mathscinet.ams.org/mathscinet-getitem?mr=4765850},
      review={\MR{4765850}},
}

\bib{Davaux_2003op}{article}{
      author={Davaux, H\'{e}l\`ene},
       title={An optimal inequality between scalar curvature and spectrum of
  the {L}aplacian},
        date={2003},
        ISSN={0025-5831},
     journal={Math. Ann.},
      volume={327},
      number={2},
       pages={271\ndash 292},
         url={https://mathscinet.ams.org/mathscinet-getitem?mr=2015070},
      review={\MR{2015070}},
}

\bib{Gilbarg_2001el}{book}{
      author={Gilbarg, David},
      author={Trudinger, Neil~S.},
       title={Elliptic partial differential equations of second order},
      series={Classics in Mathematics},
   publisher={Springer-Verlag, Berlin},
        date={2001},
        ISBN={3-540-41160-7},
         url={https://mathscinet.ams.org/mathscinet-getitem?mr=1814364},
        note={Reprint of the 1998 edition},
      review={\MR{1814364}},
}

\bib{Lee_2019ge}{book}{
      author={Lee, Dan~A.},
       title={Geometric relativity},
      series={Graduate Studies in Mathematics},
   publisher={American Mathematical Society, Providence, RI},
        date={2019},
      volume={201},
        ISBN={978-1-4704-5081-6},
         url={https://mathscinet.ams.org/mathscinet-getitem?mr=3970261},
      review={\MR{3970261}},
}

\bib{Su_2022no}{article}{
      author={Su, Guangxiang},
      author={Wang, Xiangsheng},
      author={Zhang, Weiping},
       title={Nonnegative scalar curvature and area decreasing maps on complete
  foliated manifolds},
        date={2022},
        ISSN={0075-4102},
     journal={J. Reine Angew. Math.},
      volume={790},
       pages={85\ndash 113},
         url={https://mathscinet.ams.org/mathscinet-getitem?mr=4472869},
      review={\MR{4472869}},
}

\bib{Schoen_1979pp}{article}{
      author={Schoen, Richard},
      author={Yau, Shing~Tung},
       title={On the proof of the positive mass conjecture in general
  relativity},
        date={1979},
        ISSN={0010-3616},
     journal={Comm. Math. Phys.},
      volume={65},
      number={1},
       pages={45\ndash 76},
         url={https://mathscinet.ams.org/mathscinet-getitem?mr=526976},
      review={\MR{526976}},
}

\bib{Witten_1981ne}{article}{
      author={Witten, Edward},
       title={A new proof of the positive energy theorem},
        date={1981},
        ISSN={0010-3616,1432-0916},
     journal={Comm. Math. Phys.},
      volume={80},
      number={3},
       pages={381\ndash 402},
         url={https://mathscinet.ams.org/mathscinet-getitem?mr=626707},
      review={\MR{626707}},
}

\bib{Wang_2026lm}{article}{
      author={Wang, Jian},
      author={Wang, Jinmin},
      author={Xie, Zhizhang},
       title={${L}^\infty$-metrics on tori and {S}choen's conjecture},
        date={2026},
      eprint={2606.21325v1},
         url={https://arxiv.org/abs/2606.21325v1},
}

\end{biblist}
\end{bibdiv}

\end{document}